\documentclass[a4paper,11pt]{article}

\usepackage{amssymb}
\usepackage[english]{babel}
\usepackage{graphicx}
\usepackage{xcolor,colortbl}
\usepackage{amsthm}
\usepackage{amsmath}
\usepackage{comment}

\newtheorem{theorem}{Theorem}[section]
\newtheorem{lemma}[theorem]{Lemma}
\newtheorem{conjecture}[theorem]{Conjecture}
\newtheorem{question}[theorem]{Question}

\newtheorem{proposition}[theorem]{Proposition}

\theoremstyle{definition}

\newtheorem{example}[theorem]{Example}

\newtheorem{definition}{Definition}[section]
\newtheorem{observation}[theorem]{Observation}

\begin{document}

\title{{\bf The constant objective value property for 
combinatorial optimization problems}}
\author{
\sc Ante \'{C}usti\'{c}\thanks{{\tt custic@opt.math.tugraz.at}.
Institut f\"ur Optimierung und Diskrete Mathematik, TU Graz, 
Steyrergasse 30, A-8010 Graz, Austria}
\and
\sc Bettina Klinz\thanks{{\tt klinz@opt.math.tugraz.at}.
Institut f\"ur Optimierung und Diskrete Mathematik, TU Graz, 
Steyrergasse 30, A-8010 Graz, Austria}
}
\date{July 2014}
\maketitle

\begin{abstract}
Given a combinatorial optimization problem, we aim at characterizing 
the set of all instances for which every feasible solution has the 
same objective value. 

Our central result deals with multi-dimensional assignment problems. 
We show that for the axial and for the planar $d$-dimensional assignment
problem instances with constant objective value property are characterized by 
sum-decomposable arrays. We provide a counterexample to show that the result
does not carry over to general $d$-dimensional assignment problems.

Our result for the axial $d$-dimensional assignment problem can be shown 
to carry over to the axial $d$-dimensional transportation problem. Moreover, we 
obtain characterizations when the constant objective value property 
holds for the minimum spanning tree
problem, the shortest path problem and the minimum weight maximum cardinality 
matching problem.

\end{abstract}


\section{Introduction}\label{sec:intro}

In this paper we deal with combinatorial optimization problems 
of the following type.  We are given a ground set $E=\{1,\ldots,n\}$, 
a real cost vector
$C=(c(1),\ldots,c(n))$ and a set of feasible solutions 
$\mathcal{F}\subseteq 2^{\{1,\ldots,n\}}$. The \emph{objective value\/} of a feasible solution 
$F\in \mathcal{F}$ is given by the so-called \emph{sum objective function}
\[c(F):=\sum_{i\in F}c(i).\]
The goal is to find a feasible solution $F^*$ such that $c(F^*)$ is minimal. The traveling salesman problem,  the linear assignment problem,  the
shortest path problem, Lawler's quadratic assignment problem and many other
well-known combinatorial optimization problems  fall into this class 
of problems.

\begin{definition}
We say that an instance of a combinatorial optimization problem has the 
\emph{constant objective value property} (COVP) if every feasible solution has 
the same objective value.
\end{definition}

Our goal is to
characterize the set of instances with the COVP, or in other words,
the space of all  cost vectors for which every feasible solution has 
the same objective value, for various combinatorial optimization problems.
\medskip

{\bf Related results.}
The constant objective value property is closely connected to the notion of 
\emph{admissible transformations} introduced in 1971 by Vo-Khac~\cite{V71}. 

\begin{definition}
 A transformation $T$ of the cost vector $C$ to the new cost vector 
$\tilde{C}=(\tilde{c}(1),\tilde{c}(2),\ldots,\tilde{c}(n))$ is called 
admissible with index $z(T)$, if
\[c(F)=\tilde{c}(F)+z(T)\quad \text{for all }F\in\mathcal{F}.\]
\end{definition}

Note that admissible transformations preserve the relative order of the 
objective values of all feasible solutions. 
It is well known that admissible transformations can be used as optimality
criterion and to obtain lower bounds which are useful for hard combinatorial
optimization problems. Namely, consider the combinatorial optimization problem $\min_{F\in{\cal F}} c(F)$.
Let $T$ be an admissible transformation with index $z(T)$ from the original cost vector $C$ to
the new cost vector $\tilde{C}$ such that there exists a 
feasible solution $F^*$ with the following properties:
\begin{itemize}
	\item[(i)] $\tilde{c}(i)\geq 0$\ \ \  for all $i\in\{1,\ldots,n\}$,
	\item[(ii)] $\tilde{c}(F^*)=0$.
\end{itemize}
Then $F^*$ is an optimal solution with objective value $z(T)$. 
If the condition (ii) is not satisfied or we cannot prove that it holds, 
then $z(T)$ gives a lower bound.

For the class of combinatorial optimization problems with sum objective function
there is a one-to-one correspondence between admissible 
transformations that transform the cost vector 
$(c(1),\ldots,c(n))$ to $(\tilde{c}(1),\ldots,\tilde{c}(n))$, 
and cost vectors $B=(b(1),b(2),\ldots,b(n))$ that fulfill the  COVP. 
The correspondence  is obtained by 
$c(i)=\tilde{c}(i)+b(i)$ for all $i$. Then the index of the
corresponding admissible transformation is $z(T)=\sum_{i\in F}b(i)$ 
for any $F\in\mathcal{F}$. The correspondence between the COVP and admissible transformations provides a
further source of motivation for investigating COVP characterizations.

The notion of admissible transformations can be generalized to the algebraic 
setting and applied to a wider class of combinatorial optimization problems,
including the case of bottleneck objective functions, see \cite{BZ82}.
Note, however, that for the bottleneck objective function, which is given by
$c(F)=\max_{i\in F}c(i)$, there is no one-to-one correspondence between the
COVP and admissible transformations.

Berenguer~\cite{B79} characterized the set of all admissible transformations 
for the travelling salesman problem (TSP) and the multiple salesmen version. 
All admissible transformations for the TSP are obtained by adding real 
values to rows and columns of the distance matrix.
In view of the correspondence mentioned above this result can be rephrased 
as a result on the COVP for the TSP as follows (this has been noted already
by Gilmore, Lawler and Shmoys~\cite{GLS85}).

An $n\times n$ real matrix $C=(c_{ij})$ is called {\em sum matrix\/}
if  there exist two real $n$-dimensional 
vectors $U=(u_i)$ and $V=(v_i)$ such that 
\begin{equation}\label{summatrix}
c_{ij}=u_i+v_j     \qquad\mbox{for all~~} i,j\in\{1,\ldots,n\}.
\end{equation}

\begin{theorem}[Berenguer~\cite{B79}, 
Gilmore et al.~\cite{GLS85}]\label{thm:tsp}

The TSP instance with the $n\times n$ cost matrix $C=(c_{ij})$  
has the COVP if and only if $C$ is a sum matrix.
\end{theorem}

For the TSP the diagonal entries of $C$ do not play a role and can be ignored.
Berenguer's proof works for the linear assignment problem as well, i.e.\
an instance of the linear assignment problem with cost matrix $C=(c_{ij})$ 
has the COVP if and only if $C$ is a sum matrix.

Some classes of admissible transformations for different types of 
assignment problems are listed by Burkard~\cite{B07}. However, no COVP 
characterizations are provided.

We remark that there is a simpler way to prove the COVP characterization for the
linear assignment problem mentioned above by making use of
the LP-duality and the complementary slackness condition. Since for each pair 
$(i,j)$ there exists an assignment which assigns $i$ to $j$ (i.e.\@ the primal
assignment variable $x_{ij}$ is 1), all dual constraints need to be fulfilled
with equality which is equivalent to the condition \eqref{summatrix} 
(note that the vectors $U$ and $V$ contain the dual variables). 
\medskip

{\bf Results and organization of the paper.}\@
In Section~\ref{section:assignment} we investigate the problem of 
characterizing the instances with the COVP for multi-dimensional assignment 
problems. We show that for the multi-dimensional
axial and planar case the cost arrays with the COVP
are precisely the class of sum-decomposable arrays which are generalizations of sum matrices (for the precise definition see Section~\ref{section:assignment}).
We furthermore provide a counterexample which shows that sum-decomposability is
not necessarily required for the COVP to hold for general multi-dimensional 
assignment problems.

In Section~\ref{section:transportation} the result for the 
axial $d$-dimensional assignment problem is carried over to the 
axial $d$-dimensional transportation problem. 
Finally, in Section~\ref{section:other} we 
deal with COVP characterizations for the minimum spanning tree problem, the
shortest path problem and the minimum weight maximum cardinality 
matching problem.


\section{The COVP for $d$-dimensional 
assignment\\ problems}
\label{section:assignment}

Berenguer's result for the classical linear assignment problem motivated us to
ask for COVP characterizations for multi-dimensional assignment problems.

\subsection{General multi-dimensional assignment problems} 

Two classical ways of generalizing the notion of assignments to 
three dimensions
are the so-called \emph{axial 3-dimensional \textup{(}or 3-index\textup{)} assignment problem} 
and the \emph{planar 3-dimensional \textup{(}or 3-index\textup{)} assignment problem}.
A further generalization is obtained by the class of $d$-dimensional assignment
problems defined in the sequel. For more on the topic of assignment problems 
see the book by Burkard et al.~\cite{BDM09} and the references cited therein.

A general multi-dimensional assignment problem is specified by two 
parameters $d$ and $s$, where $d$ is the number of indices and $s$ describes 
the number of fixed indices in the constraints. 
Informally speaking, we want to  find a set of $n^s$
elements of a $d$-dimensional $n\times n\times \cdots \times n$ array $C$ with 
minimal total sum, such that for every 
$s$ fixed indices of $C$ exactly one element
is chosen. 

Formally, the $(d,s)$ assignment problem, $(d,s)$-AP for short, can be 
stated in the following way.

\begin{definition}
Let $d$ and $s$ be integers with $0<s<d$. The input of the 
$(d,s)$-AP consists of an integer $n\ge 1$ and a $d$-dimensional 
$n\times n\times \cdots \times n$ cost array $C$ which associates 
the cost $c(i_1,i_2,\ldots,i_d)$ to the $d$-tuple 
$(i_1,i_2,\ldots,i_d)\in \{1,\ldots,n\}^d$. Let
$\mathcal{Q}_s$ be the set of all subsets of $K=\{1,\ldots,d\}$ 
with cardinality $s$, 
i.e.\@ $\mathcal{Q}_s=\{Q\colon Q\subset K, |Q|=s\}$. 
For any set $Q=\{q_1,q_2,\ldots,q_s\}\in\mathcal{Q}_s$ of fixed indices with
$q_1<q_2< \cdots <q_s$ and any $s$-tuple
$t=(t_1,\ldots,t_s)\in\{1,\ldots,n\}^s$, let $T(Q,t)$ be the set of all 
$d$-tuples $t'=(t'_1,\ldots,t'_d)\in \{1,\ldots,n\}^d$
such that $t'_{q_j}=t_j$ for all $j=1,\ldots,s$.
The general $(d,s)$ assignment problem $(d,s)$-AP can be stated 
as
\begin{align}
	\min\sum_{i_1=1}^n\cdots\sum_{i_d=1}^n c(i_1,i_2,\ldots,i_d) x(&i_1,i_2,\ldots,i_d) \label{dsAP}\\
	\text{s.t. } \sum_{\substack{(i_1,\ldots,i_d)\in\\  T(Q,(j_1,\ldots,j_{s}))}} x(i_1,i_2,\ldots,i_d)=1 &\text{ for all } Q\in \mathcal{Q}_s\nonumber\\[-2em]
	&   \text{ and all }(j_{1},\ldots,j_{s})\in \{1,\ldots,n\}^s,\nonumber\\[0.5em]
	x(i_1,i_2,\ldots,i_d)\in\{0,1\}\   &\text{ for all } (i_1,i_2,\ldots ,i_d)\in \{1,\ldots,n\}^d.\nonumber
\end{align}
\end{definition}
Note that in each of the equality constraints above, the sum essentially extends over $d-s$
variables (corresponding to the free indices from the  set $K\setminus Q$).

Let $X=(x(i_1,i_2,\ldots,i_d))$ be a feasible solution of the integer program
stated above. Then the set $F=\{(i_1,i_2,\ldots,i_d)\colon x(i_1,i_2,\ldots,i_d)=1\}$ is a 
feasible solution of the $(d,s)$-AP and the value 
$\sum_{(i_1,\ldots,i_d)\in F}c(i_1,\ldots,i_d)$ is the cost (objective value)
of a feasible solution $F$. Hence the $(d,s)$-AP fits into the 
class of combinatorial optimization problems with sum objective function.

Using the $(d,s)$-AP notation, the classical linear assignment problem is the 
$(2,1)$-AP, 
while the axial and the planar 3-dimensional assignment problems correspond to the 
$(3,1)$-AP and the $(3,2)$-AP, respectively. More generally, we refer to the 
$(d,1)$-AP as \emph{axial $d$-dimensional assignment problem}, and to the 
$(d,d-1)$-AP as \emph{planar $d$-dimensional assignment problem}. 
Let us remark that for $d\ge 4$, there is no consensus in the literature 
which problem version is referred to as planar $d$-dimensional assignment 
problem. Our decision to refer to the $(d,d-1)$-AP in which all constraints
involve single sums as to planar is in accordance with the axial/planar 
nomenclature which was introduced in the original paper by 
Schell~\cite{Sch55} and is used by 
Spieksma~\cite{Sp00}, Frieze and Sorkin~\cite{FS13} and others. 
A group of authors around Appa, see e.g.\@ \cite{AMM06b},
refer to the $(d,2)$-AP as to planar $d$-dimensional assignment 
problem. The axial and the planar 3-dimensional assignment problems are known to be 
NP-hard~\cite{F83,K72}. As a consequence thereof both the axial and the 
planar $d$-dimensional assignment problems are NP-hard for all $d\geq 3$.

The following observations collect a few well known facts 
about the structure of the set of feasible solutions of 
the $(d,s)$-AP. These will turn out to be helpful in later parts of the paper.

\begin{observation}\label{obs1}
Every feasible solution of the $(d,1)$-AP of size $n$ can be represented by a set of $d-1$ permutations of $\{1,\ldots,n\}$. Namely, every feasible
solution can be written as $F=\{(i,\phi_1(i),\ldots,\phi_{d-1}(i))\colon i=1,\ldots,n\}$ where each $\phi_k$ is a permutation of $\{1,\ldots,n\}$.
\end{observation}

\begin{observation}\label{obs2}
A set of $d$-tuples $F$ is a feasible solution of the 
$(d,d-1)$-AP of size $n$ if and only if $F$ ``contains" $n$ pairwise
disjoint feasible solutions of the $(d-1,d-2)$-AP. Namely, define $n$ sets of $(d-1)$-tuples obtained from $F$ by fixing one
 index, for example the first one: $F_i=\{(a_2,a_3,\ldots,a_d) \colon
 (i,a_2,a_3,\ldots,a_d)\in F\}$ for $i=1,2,\ldots, n$. Then every $F_i$ is a
 feasible solution of the  $(d-1,d-2)$-AP because fixing $d-2$ indices in $F_i$
 corresponds to fixing $d-1$ indices in $F$. Also, if there are $F_i$ and $F_j$
 that are not disjoint, then there would be two elements in $F$ that coincide on
 $d-1$ indices, a contradiction. The same construction 
works in the other direction as well.
\end{observation}

\begin{observation}\label{obs3}
The feasible solutions of the $(d,2)$-AP for $d\ge 3$ correspond to 
$(d-2)$-tuples of mutually orthogonal Latin squares. Indeed, assume that $F$
is a feasible solution of the $(d,2)$-AP. We can 
represent $F$ as an $n\times n$ table $T$ with $(d-2)$-tuples as entries 
in the following way: The $(d-2)$-tuple $(i_3,\ldots,i_d)$ is the entry in 
row $i_1$ and column $i_2$ of $T$ if and only if $(i_1,\ldots,i_d)$ is an 
element of $F$. Since $F$ is a feasible solution of the $(d,2)$-AP, 
each row and each column of $T$ contain every integer from $1$ to $n$ exactly
once on the $k$-position for all $k=1,\ldots,d-2$. Moreover, each $(d-2)$-tuple
of pairwise distinct integers from $\{1,\ldots,n\}$ appears exactly once in
$T$. Hence $T$ can be interpreted as a $(d-2)$-tuple of mutually orthogonal
Latin squares (the $k$-th component of the entries of $T$ yields the 
$k$-th Latin square). Note that if $d=4$ then $T$ is a Graeco-Latin square.
\end{observation}

What distinguishes the general $(d,s)$-AP from the special cases with $s=1$
(axial problem) and with $s=d-1$ (planar problem) is that there does not need
to exist feasible solutions for every value of $n$. Furthermore, not much is known
on the structure of the set of feasible solutions for the $(d,s)$-AP for
general $n$, cf.~Appa et al.~\cite{AMM06b}.
Infeasible instances and instances with very few feasible
solutions provide clear obstacles to our intended COVP characterization. For this reason the feasibility topic
for the $(d,s)$-AP plays a role for us and we briefly review a few basic
results from the literature.

The question for which values $n$ the general $(d,s)$-AP has feasible solutions
is a very difficult problem which is still open for many combinations of 
$d$ and $s$, and is related to a number of difficult problems in combinatorics. More
precisely, the general $(d,s)$-AP of size $n$ has a feasible solution 
if there exists an $s$-transversal design with $d$ groups of size $n$, or
equivalently if there exists an orthogonal array OA($n,d,s$) of index 1, strength
$s$ and order $n$, see \cite{Co06}.

In view of Observation~\ref{obs3}, the question of existence of Graeco-Latin squares is of interest. This was a famous open problem for a long time going 
back to Euler until it was finally proved by
Bose et al.~\cite{BSP60} that Graeco-Latin squares, and thus feasible
solutions of the $(4,2)$-AP, exist for every $n\ge 3$ except for $n=6$.

The number of mutually orthogonal Latin squares of size $n$ (see Observation~\ref{obs3} for the
connection to the feasibility of the $(d,2)$-AP) is unknown for general $n$. It
is known however that this number is at most $n-1$ and that the upper bound is 
achieved if $n$ is a prime power. It is also known that there exist $n-1$
mutually orthogonal Latin squares if and only if there exists a projective
plane of order $n$, see \cite{Co06}.

\subsection{Sum-decomposable arrays}

In this subsection we investigate the vector spaces of 
\emph{sum-decomposable arrays}. These will occur as solutions of various COVP
characterizations. Sum-decomposable arrays generalize the concept of
sum matrices to higher dimensions.

Informally, a $d$-dimensional $n\times n \times \cdots  \times n$
 real array $C$ is sum-de\-compo\-sable with parameters $d$ and $s$ 
(and size $n$)
 if $C$ can be obtained as a sum of ${d \choose s}$ $s$-dimensional arrays, one
 for each subset of $\{1,\ldots,d\}$ of size $s$.
For example, in the case $d=3$ and $s=2$, $C=(c_{ijk})$ is  sum-decomposable 
if there exist three two-dimensional real arrays $A=(a_{ij})$, $B=(b_{ij})$ and 
$D=(d_{ij})$ such that $c_{ijk}=a_{ij}+b_{ik}+d_{jk}$. A formal definition follows.
\begin{definition}
Let $n$, $d$ and $s$ be integers such that $d>s>0$ and $n>1$. 
Let  
$\mathcal{Q}_s=\{Q\colon Q\subset \{1,\ldots,d\}, |Q|=s\}$. 
Then the $d$-dimensional 
$n\times n\times\cdots\times n$ real array $C$ is 
called {\em sum-decomposable with parameters $d$ and $s$ and size $n$} 
if there exist ${d \choose s}$ $s$-dimensional 
$n\times n\times\cdots\times n$ real arrays $A^{Q}=(a^Q(j_1,\ldots,j_s))$, one for each 
$Q\in \mathcal{Q}_s$, such that  \begin{equation*}
c(i_1,i_2,\ldots,i_d)=\sum_{Q\in \mathcal{Q}_s} 
a^{Q}(h_{Q}(i_1,i_2,\ldots,i_d))
\end{equation*}
where  
$h_{Q}(i_1,i_2,\ldots,i_d)$ denotes the $s$-tuple associated with $Q$, i.e.\@ 
$h_{Q}(i_1,i_2,\linebreak\ldots,i_d)=(i_{q_1},\ldots, i_{q_s})$ for 
$Q=\{q_1,\ldots$ $,q_s\}$, $q_1<q_2<\cdots <q_s$.
\end{definition}

We denote the vector space of all sum-decomposable real arrays of size $n$ with 
parameters $d$ and $s$ by {\sc SAVS}($d,s,n$).

For $Q=\{j_1,j_2,\ldots,j_s\}\in\mathcal{Q}_s$ 
let $V_Q$ denote the vector space of all $d$-dimensional 
$n\times n \times \cdots  \times n$ 
arrays $C=(c(i_1,i_2,\ldots,i_d))$ for which
there exists a mapping  $f\colon \{1,2,\ldots , n\}^s \mapsto \mathbb{R}$ 
with $c(i_1,i_2,\ldots,i_d)=f(i_{j_1},i_{j_2},\ldots,i_{j_s})$. In other words,
the value $c(i_1,i_2,\ldots,i_d)$ depends only on the $s$ indices 
from the set $Q$ and not on all $d$ indices.
Let $Q_1, Q_2, \ldots, Q_{d \choose s}$ be such that $\mathcal{Q}_s=\{Q_1,\ldots,Q_{d \choose s}\}$. Note that
\begin{equation}\label{savs}
\text{{\sc SAVS}}(d,s,n)=V_{Q_1}+V_{Q_2}+\cdots +V_{Q_{d \choose s}}.
\end{equation}

We will use the following proposition, that can be found in \cite[Prop.~7.1 of
Chap.~1, p.~15]{PP05}, to prove that a variant of inclusion-exclusion 
principle holds for the vector subspaces $V_{Q_{i}}$. 
Recall that the distributivity with respect to sum and intersection does not hold for arbitrarily vector subspaces, i.e.\@ it is not true that $V_1 \cap \left(V_2 + V_3\right) = \left(V_{1} \cap V_{2}\right) + \left(V_{1} \cap
V_{3}\right)$ holds for all vector spaces $V_1, V_2, V_3$.

\begin{proposition}\label{vspaces}
Let $W$ be a vector space and $V_1, V_2, \ldots, V_n\subset W$ be a collection of its subspaces. Then the following conditions are equivalent:
\begin{itemize}
\item[(i)] The collection $V_1, V_2, \ldots, V_n$ is distributive with respect to the operations of sum and intersection.
\item[(ii)] There exists a basis $\{w_{\alpha}:\alpha\in A\}$ of the vector space $W$ such that each of the subspaces $V_i$ is the linear span of a set of vectors $w_{\alpha}$.
\end{itemize}
\end{proposition}

Now we are ready to prove the following proposition.

\begin{proposition}\label{prop:SAVS}
Let {\text{\sc SAVS}}$(d,s,n)$ be expressed as in~\eqref{savs}. Then we have
\begin{itemize}
\item[(i)]$\dim(V_Q)=n^s$ for all $Q\in\mathcal{Q}_s$,

\item[(ii)] 
$\dim\left(   \cap_{i\in I} V_{Q_i} \right)  = n^{| \cap_{i\in I} Q_i |} 
\text{~for all~} I\subset \{1,\ldots, {d \choose s}\},$
\item[(iii)] 
$\displaystyle \dim({\text{\sc SAVS}}(d,s,n))=\dim\left(\sum_{i=1}^{{d \choose
        s}}V_{Q_i}\right)=$

$\displaystyle \sum_{k = 1}^{{d \choose s}} (-1)^{k+1} 
\left( \sum_{1 \leq i_{1} < \cdots <
    i_{k} \leq {d\choose s}} \dim\left( V_{Q_{i_{1}}} \cap \cdots \cap
    V_{Q_{i_{k}}} \right) \right), 
$

\item[(iv)] 
$\dim({\text {\sc SAVS}}(d,d-1,n)=n^d-(n-1)^d,$

$\dim({\text{\sc SAVS}}(d,1,n)=dn-d+1$.
	\end{itemize}
\end{proposition}

\begin{proof} 

Ad (i): Follows directly from the definition of $V_Q$.

Ad (ii): Let us start with $|I|=2$ and consider $J=\{j_1,j_2,\ldots,j_s\}$, 
$K=\{k_1,k_2,\ldots,k_s\}$ from  $\mathcal{Q}_s$. Further, let 
$C$ and $E$ be two  arrays from $V_J$ and $V_K$, respectively. Hence,
$c(i_1,i_2,\ldots,i_d)=f(i_{j_1},i_{j_2},\ldots,i_{j_s})$ and
$e(i_1,i_2,\ldots,i_d)=g(i_{k_1},i_{k_2}, \ldots, i_{k_s})$ for some $f,g
\colon \{1,2,\ldots , n\}^s \mapsto \mathbb{R}$. Then for every
$A=(a(i_1,\ldots,i_d))\in V_J\cap V_K$ we have that
$a(i_1,i_2,\ldots,i_d)=t(i_{q_1},\ldots,i_{q_{|J\cap K|}})$, $q_i\in J\cap K$,
for some $t\colon \{1,2,\ldots, n\}^{|J\cap K |} \mapsto \mathbb{R}$. Hence,
$\dim(V_J\cap V_K)=n^{|J\cap K |}$. The case $|I|\ge 3$ follows by an inductive 
argument. This settles (ii).

Ad (iii): The case of two vector spaces involved in the sum 
follows from the fact that
\begin{equation}\label{eq:case2}
\dim (V_1+V_2)=\dim(V_1)+\dim(V_2)-\dim(V_1\cap V_2)
\end{equation}
holds for any two subspaces $V_1$ and $V_2$ of a vector space. Next we settle 
the case of three vector spaces, so let $Q_i, Q_j, Q_\ell\in \mathcal{Q}_s$.
The general case then follows by induction.
Note that (\ref{eq:case2}) implies that
\begin{align} \label{eq:case3}
\dim \left(V_{Q_i}+V_{Q_j}+V_{Q_\ell}\right) &= \nonumber\\
\dim \left(V_{Q_i}\right)+\dim&\left(V_{Q_{j}}+V_{Q_\ell}\right)-\dim\left(V_{Q_i}\cap \left(V_{Q_j}+V_{Q_\ell}\right)\right). 
\end{align}
To proceed further it suffices to prove that the subspaces $V_{Q_i}$ are 
distributive with respect to sum and intersection, i.e.\@ that
\begin{equation}\label{eq:distr}
V_{Q_i} \cap \left(V_{Q_j} + V_{Q_\ell}\right) = \left(V_{Q_i} \cap V_{Q_j}\right) + \left(V_{Q_i} \cap
V_{Q_\ell}\right)
\end{equation}
holds. It is easy to check that using (\ref{eq:distr}) and (\ref{eq:case2}) in 
(\ref{eq:case3}) above leads to the claim (iii) for three vector spaces. So, it only remains to show that the distributivity property \eqref{eq:distr} holds. To that end, we construct bases for the vector spaces $V_Q$ for $Q\in
{\mathcal Q_s}$ and for the vector space $V$ of all $d$-dimensional 
$n\times \cdots \times n$ real arrays. The distributivity then follows 
from Proposition~\ref{vspaces}. Namely, call an array from $V$ {\em elementary} if a single entry is 1 and all other
entries are 0. It is easy to see that the set of $n^d$ 
$d$-dimensional $n\times n \times \cdots  \times n$ elementary arrays
forms a basis for $V$. Next we construct a basis for the subspace $V_Q$ where
$Q=\{j_1,j_2,\ldots,j_s\}$. Let $A_{k_1,\ldots, k_s}$ be the 0-1 
$d$-dimensional array such that the entry at position $(i_1, \ldots, i_d)$ is 1
if $i_{j_1}=k_1$, $\ldots$, $i_{j_s}=k_s$ and 0 otherwise. 
Then the set of arrays $\{A_{k_1,\ldots, k_s} \colon (k_1,\ldots, k_s)\in
\{1,\ldots,n\}^s\}$ forms a basis for $V_Q$. 
Note that every element of the basis for $V_Q$ can be written as a linear
combination of elementary arrays.

Ad (iv): Note that any intersection of $\ell$ distinct subsets of 
$\{1,2,\ldots,d\}$ has cardinality $d-\ell$. Hence, from (ii) and (iii) 
it follows that
\[\dim({\text{\sc SAVS}}(d,d-1,n))=\sum_{i=1}^d(-1)^{i+1}{d \choose i}n^{d-i},\]
which is equal to $n^d-(n-1)^d$ by the binomial theorem. Since the intersection of any distinct one-element sets is empty it follows that
$\dim({\text{\sc SAVS}}(d,1,n))=dn-d+1$. 
\end{proof}

\subsection{The COVP for the axial case: ($d,1$)-AP}

Now we turn to the problem of characterizing the instances of the axial
$d$-dimensional assignment problem with the constant objective value property (COVP).

\begin{theorem}\label{thm:axial}
An instance of the $(d,1)$-AP with cost array $C$ has the COVP if and only if $C$ is a 
sum-decomposable array with parameters $d$ and $1$.
\end{theorem}

\begin{proof}Note that one direction follows immediately, i.e.\@ if $C$ is the sum of $d$ vectors, then every feasible solution has the same objective value. Conversely, assume that every feasible solution has the same objective value. For integers $i_1,i_2,\ldots,i_d\in\{2,3,\ldots,n\}$ consider the following $d$ pairs of $d$-tuples:
\begin{center}
\begin{tabular}{l}
$(1,1,\ldots ,1)$,\ \ \ $(i_1,i_2,\ldots,i_d)$\\
$(i_1,1,\ldots,1)$,\ \ \ $(1,i_2,\ldots,i_d)$\\
$(1,i_2,1,\ldots,1)$,\ \ \ $(i_1,1,i_3,\ldots,i_d)$\\
	\qquad$\vdots$\\
$(1,\ldots,1,i_d)$,\ \ \ $(i_1,\ldots,i_{d-1},1).$
\end{tabular}\end{center}
There exists a set of $n-2$ $d$-tuples which completes each of these pairs to a feasible solution; for example the set $\{(k_1^j,k_2^j,\ldots,k_d^j)\colon j=2,\ldots,n-1,\  k_l^j=j \text{ if } j<i_l \text{ and }k_l^j=j+1 \text{ otherwise},\  l=1,\ldots d\}$. By assumption we have
\begin{align}
c(i_1,i_2,\ldots,i_d)&=c(i_1,1,\ldots,1)+c(1,i_2,\ldots,i_d)-c(1,\ldots ,1) \label{axial:first}\\
	&=c(1,i_2,1,\ldots,1)+c(i_1,1,i_3,\ldots,i_d)-c(1,\ldots ,1) \label{axial:second}\\
	&\hspace{6pt}\vdots\notag\\
	&=c(1,\ldots,1,i_d)+c(i_1,\ldots,i_{d-1},1)-c(1,\ldots ,1).\notag
\end{align}
Due to \eqref{axial:first} there exist a vector $V_1=(v_1(i))$
and a $(d-1)$-dimensional array $G_1=(g_1(i_1,\ldots,i_{d-1}))$ such that
$c(i_1,i_2,\ldots,i_d)=v_1(i_1)+g_1(i_2,\ldots,i_d)$. 
Analogously, from~\eqref{axial:second} it follows that there 
exists a vector $V_2$ and a $(d-1)$-dimensional array $G_2$ such that 
$c(i_1,i_2,\ldots,i_d)=v_2(i_2)+g_2(i_1,i_3,\ldots,i_d)$. Hence, $c(i_1,i_2,\ldots,i_d)=v_1(i_1)+v_2(i_2)+g_{1,2}(i_3,\ldots,i_d)$ 
for some $(d-2)$-dimensional array $G_{1,2}=(g_{1,2}(i_1,$ $\ldots,i_{d-2}))$. 
Using the remaining equations in an analogous manner we finally obtain
that $C$ is the sum of $d$ vectors, i.e.
\[c(i_1,i_2,\ldots,i_d)=v_1(i_1)+v_2(i_2)+\cdots+v_d(i_d),\]
where the vectors $V_k=(v_k(i))$ can be chosen as follows:
\begin{align*}
	v_1(i)&=c(i,1,\ldots,1)-\frac{d-1}{d}c(1,1,\ldots,1),\\
	&\hspace{6pt} \vdots\\
	v_d(i)&=c(1,\ldots,1,i)-\frac{d-1}{d}c(1,1,\ldots,1).
\end{align*}
\end{proof}

\subsection{The COVP for the planar case: ${(d,d-1)}$-AP}

We now turn to the planar case. Note that there are exactly two feasible 
solutions of the $(d,d-1)$-AP when $n=2$.

\begin{definition}
We say that an instance of the $(d,d-1)$-AP with cost array $C$ 
has property $P_2$ if for every $2\times2\times\cdots \times2$ sub-array of 
$C$, which is obtained by restricting the index sets to $\{1,i_1\}\times
\{1,i_2\}\times\cdots\times \{1,i_d\}$,
the two feasible solutions on the resulting subproblem of size 2
have the same objective value.
\end{definition}

Property $P_2$  and sum-decomposable cost arrays for the $(d,d-1)$-AP are
related in the following way.

\begin{lemma}\label{lem:planar}
Let $I$ be an instance of the  $(d,d-1)$-AP with cost array $C$. If $I$ has property $P_2$, then $C$ is a sum-decomposable array with
parameters $d$ and $d-1$.
\end{lemma}

\begin{proof}
Consider the $2\times2\times\cdots \times2$ subarray $D_2$ of $C$ obtained by 
restricting index sets to $\{1,i_1\}\times \{1,i_2\}\times\cdots \times\{1,i_d\}$ with $i_j\in
\{2,\ldots,n\}$ for $j=1,\ldots,d$. By exploiting the fact that the two feasible solutions for the subarray $D_2$ have 
the same objective value  we get that
\begin{equation}\label{prva}
	c(i_1,i_2,\ldots.i_d)=\sum_{x\in I_1}c(x)-\sum_{x\in I_2}c(x)+\cdots +(-1)^{d+1}\sum_{x\in I_d}c(x),
\end{equation}
where $I_i$ is the set of all $d$-tuples from $\{1,i_1\}\times \{1,i_2\}\times\cdots \times\{1,i_d\}$ with exactly $i$ ones.
Then from~\eqref{prva} it follows that $C$ can be expressed as the sum of $d$ $(d-1)$-dimensional arrays $A_j=(a_j(i_1,\ldots,i_{d-1}))$, $j=1,\ldots, d$, defined by
\begin{align*}
a_1(i_2, i_3,\ldots,i_d)=&\sum_{x\in I^1_1}c(x)-\frac{1}{2}\sum_{x\in I^1_2}c(x)+\cdots +(-1)^{d+1}\frac{1}{d}\sum_{x\in I^1_d}c(x) \\
a_2(i_1, i_3,\ldots,i_d)=&\sum_{x\in I^2_1}c(x)-\frac{1}{2}\sum_{x\in I^2_2}c(x)+\cdots +(-1)^{d+1}\frac{1}{d}\sum_{x\in I^2_d}c(x)\\
\vdots\ &\\
a_d(i_1, i_2,\ldots,i_{d-1})=&\sum_{x\in I^d_1}c(x)-\frac{1}{2}\sum_{x\in I^d_2}c(x)+\cdots +(-1)^{d+1}\frac{1}{d}\sum_{x\in I^d_d}c(x),
\end{align*}
where $I^k_i$ is the set of all $d$-tuples from $\{1,i_1\}\times \{1,i_2\}\times\cdots \times\{1,i_d\}$ with exactly $i$ ones, one of which is on the $k$-th coordinate.
\end{proof}

The following result relates property $P_2$ and the COVP. 

\begin{proposition}\label{prop:planar}
Every instance of the $(d,d-1)$-AP with cost array $C$ with $n\neq 3$ that has 
the COVP, also has property $P_2$.
\end{proposition}

\begin{proof} 
We will prove that both feasible solutions of the $(d,d-1)$-AP on the sub-array of $C$ with 
indices $\{1,2\}\times \{1,2\}\times\cdots\times\{1,2\}$ have the 
same objective value; the general case can be shown analogously.  When $n=2$, this is trivially true. Assume $n\ge 4$.  We will build two different feasible solutions $F^d_1$ and $F^d_2$ for the $(d,d-1)$-AP that satisfy the following
property: $F^d_1$ and $F^d_2$ both contain a feasible solution of the 
$(d,d-1)$-AP on the subproblem induced by the index set $\{1,2\}^d$, and 
all other elements of these two solutions are the same.
The existence of such $F^d_1$ and $F^d_2$ completes the proof. Namely, by assumption the objective values of 
$F_1^d$ and $F_2^d$ are equal, hence \eqref{prva} holds.

Next we explain how 
$F^d_1$ and $F^d_2$ can be constructed recursively from a feasible solution 
of the $(d-1,d-2)$-AP, which we denote by $F^{d-1}$, which also contains a feasible solution on the subproblem of size 2 induced by the index set $\{1,2\}^{d-1}$. We define $F^d_j$, $j=1,2$ as follows:
\[F_j^d=\{(i,a_1,a_2,\ldots,a_{d-2},\phi^j_i(a_{d-1}))\colon(a_1,\ldots,a_{d-1})
\in F^{d-1},i=1,\ldots,n\},\]
where $n$ permutations 
$\phi^j_i$, $i=1,\ldots,n$, are chosen to be mutually disjoint (recall that two permutations $\alpha$ and $\beta$ are disjoint if 
$\alpha(i)\neq\beta(i)$ for all $i$). Furthermore, for every $i$ we choose $\phi^1_i$ and $\phi^2_i$ such that they coincide except for 
$\phi^1_1(1)=1$, $\phi^1_1(2)=2$, $\phi^1_2(1)=2$, $\phi^1_2(2)=1$, in contrast 
to $\phi^2_1(1)=2$, $\phi^2_1(2)=1$, $\phi^2_2(1)=1$, $\phi^2_2(2)=2$.
To show that such two sets of permutations (for $j=1$ and $j=2$) exist, we represent them as 
two $n\times n$ Latin squares. For $j=1,2$, let the $j$-th table contain the 
integer $\phi^j_r(s)$ in the row $r$ and column $s$. The resulting tables 
will be two Latin squares of order $n$ which are identical except in the 
$2\times2$ upper-left corner. That corner is filled with two different 
Latin squares of order 2, respectively. It is well known that for 
$n\geq 4$ such Latin squares exist, see~\cite{R51}. From Observation~\ref{obs2} we get that $F_j^d$ are 
indeed feasible solutions.
\end{proof}

The approach we followed in the proof of Proposition~\ref{prop:planar} did not
serve us to cover the case $n=3$ and $d\geq 5$. Using a
linear algebra approach we were able to cover this case as well and hence 
to prove the following COVP characterization 
for the $(d,d-1)$-AP.

\begin{theorem}\label{thm:planar}
An instance of the $(d,d-1)$-AP with cost array $C$ has 
the COVP if and only if  $C$ is 
a sum-decomposable array with parameters $d$ and $d-1$.
\end{theorem}

\begin{proof}
If  the cost array $C$ is sum-decomposable, then it is straightforward to see that every feasible solution has the same objective value. 
		
Conversely, assume that every feasible solution has the same objective value.
For the case $n\ne 3$ the statement follows from Proposition~\ref{prop:planar}
and Lemma~\ref{lem:planar}. For the remaining case $n=3$ we make use of the same technique that has been
used in~\cite{GLS85,LR79} to obtain a COVP characterization for the TSP. Let ${\cal C}(d,n)$ denote the collection of all $d$-dimensional 
$n\times n\times \cdots \times n$ cost arrays $C$ for which all 
feasible solutions of the $(d,d-1)$-AP have the same 
objective value. Clearly ${\cal C}(d,n)$ is a linear subspace of the set of all 
$d$-dimensional  $n\times n\times \cdots \times n$ arrays. Our goal is to prove that 
\begin{equation}\label{cd3}
{\cal C}(d,3)={\text{\sc SAVS}}(d,d-1,3).
\end{equation}
To that end, we consider all feasible solutions of the $(d,d-1)$-AP for $n=3$. 
Next we build up the 0-1 matrix $M_d$ where the rows of $M_d$ correspond to the feasible solutions
and the columns correspond to the $d$-tuples over $\{1,2,3\}$.
The entry of $M_d$ that corresponds to the feasible solution $F$ and the 
$d$-tuple $(i_1,i_2,\ldots,i_d)$ is set to 1 if and only if  
$(i_1,i_2,\ldots,i_d)\in F$.

Note that every row of the 
matrix $M_{d+1}$ is obtained from three disjoint rows of the matrix $M_{d}$. 
For every row $r_1$ of the matrix $M_{d}$ there are exactly two rows 
$r_2$, $r_3$ disjoint with $r_1$, and $r_2$ and $r_3$ are also mutually
disjoint. Therefore $r_1r_2r_3$ and  $r_1r_3r_2$ are rows of $M_{d+1}$. Hence
the matrix $M_{d+1}$ has twice as many rows as $M_{d}$. This corresponds to the
fact that for $n=3$ the number of feasible solutions doubles when 
moving from the planar $d$-dimensional assignment problem to the
$(d+1)$-dimensional one. It is easy to see that $M_d$ is a $3\cdot 2^{d-1}\times 3^d$ matrix.
The following matrices $M_{d}$, $d=1,2$ are provided as illustration:
\smallskip

$M_1=\begin{pmatrix}
	1&0&0 \\
	0&1&0 \\
	0&0&1
\end{pmatrix}$,\  \ 
$M_2=\left( \begin{array}{ccc|ccc|ccc}
	1&0&0&0&1&0 &0&0&1\\
	0&1&0 &1&0&0&0&0&1\\
	0&0&1&1&0&0&0&1&0  \\ \hline
	1&0&0&0&0&1&0&1&0\\
	0&1&0&0&0&1&1&0&0\\
	0&0&1&0&1&0&1&0&0
\end{array}\right).$ 
\smallskip

${\cal C}(d,3)$ is the solution space of the system of linear equations with
coefficient matrix $M_d$ and a constant right hand side vector. 
Thus we obtain
\[
\dim {\cal C}(d,3)=3^d+1-\text{rank}(M_d).
\]
From Proposition~\ref{prop:SAVS}~(iv)
we know that
$\dim({\text{\sc SAVS}}(d,d-1,3))=3^d-2^d$. 
Hence in order to prove that \eqref{cd3} holds, 
we need to show that $\text{rank}(M_d)=2^d+1$.
Observe that in fact it suffices to show that $\textnormal{rank}(M_d)\ge 2^d+1$
since obviously ${\text{\sc SAVS}}(d,d-1,3)\subseteq {\cal C}(d,3)$;
Lemma~\ref{lem:rank} below completes the proof.
\end{proof}

\begin{lemma}\label{lem:rank} 
Let $M_d$ be the matrix constructed above. We have
$$\textnormal{rank}(M_d)\geq 2^d+1.$$
\end{lemma}

\begin{proof}
We start with observing the following recursive structure of $M_d$.
Define
$$A_0=\begin{pmatrix}
	1&0&0 \\
	0&1&0 \\
	0&0&1
\end{pmatrix}\qquad
B_0=\begin{pmatrix}
	0&1&0 \\
	1&0&0 \\
	1&0&0
\end{pmatrix}
\qquad
C_0=\begin{pmatrix}
	0&0&1 \\
	0&0&1 \\
	0&1&0
\end{pmatrix},$$
and recursively for $k\ge 1$
$$
A_k=\begin{pmatrix}
A_{k-1} & B_{k-1} & C_{k-1}\\
A_{k-1} & C_{k-1} & B_{k-1}\\
\end{pmatrix}\qquad
B_k=\begin{pmatrix}
B_{k-1} & C_{k-1} & A_{k-1}\\
B_{k-1} & A_{k-1} & C_{k-1}\\
\end{pmatrix}$$
$$
C_k=
\begin{pmatrix}
C_{k-1} & A_{k-1} & B_{k-1}\\
C_{k-1} & B_{k-1} & A_{k-1}\\
\end{pmatrix},
$$
where $A_k$, $B_k$ and $C_k$ are $3\cdot 2^{k}\times 3^{k+1}$ matrices.
It is easy to see that $M_d=A_{d+1}$ for $d\geq 1$.
Next we will exhibit a regular $(2^d+1)\times (2^d+1)$ 
submatrix $M'_d$ of $M_d$ which will settle the lemma.  
We construct new matrices $A'_k$, $B'_k$ and 
$C'_k$ from $A_k$, $B_k$ and $C_k$ as follows: First, remove 
all columns with indices $\ge 2\cdot 3^{k}+1$. Next, remove all rows and
columns with indices that are divisible by 3. It is straightforward to
observe that the recursive structure survives this construction.
More precisely we have
\begin{equation}\label{eq:rec}
A'_k=\begin{pmatrix}
A'_{k-1} & B'_{k-1}\\
A'_{k-1} & C'_{k-1}\\
\end{pmatrix}\
B'_k=\begin{pmatrix}
B'_{k-1} & C'_{k-1}\\
B'_{k-1} & A'_{k-1}\\
\end{pmatrix}\
C'_k=
\begin{pmatrix}
C'_{k-1} & A'_{k-1}\\
C'_{k-1} & B'_{k-1} 
\end{pmatrix}
\end{equation}
for $k\ge 1$ and 
$$A_0'=\begin{pmatrix}
	1&0 \\
	0&1
\end{pmatrix}\qquad
B'_0=\begin{pmatrix}
	0&1\\
	1&0
\end{pmatrix}
\qquad
C'_0=\begin{pmatrix}
	0&0 \\
	0&0
\end{pmatrix}.$$
The matrices $A'_k$, $B'_k$ and $C'_k$ have $2^{k+1}$ rows and $2^{k+1}$ columns. We obtain our target matrix $M'_d$ from the matrix 
$A'_{d+1}$ by re-inserting row 3 and column 3 of the matrix $A_{d+1}$. 
In order to show  that $M'_d$ is regular, we will
calculate its determinant by a recursive approach. We will make use of the observation that the upper left 
and lower left block are identical in the matrices 
$A'_k$, $B'_k$ and $C'_k$. This will allow us to create a zero block as 
lower left block of a reduced matrix which has the same determinant as 
$A'_k$. 
This results in
\begin{equation}\label{eq:det1}
\det A'_k=\det A'_{k-1}\det\left(C'_{k-1}-B'_{k-1}\right)
\end{equation}
for $k\ge 1$. An analogous argument yields
\begin{equation}\label{eq:det2}
\det\left(C'_{k}-B'_{k}\right)=
\det\left(C'_{k-1}-B'_{k-1}\right)
\det\left(B'_{k-1}+C'_{k-1}-2A'_{k-1}\right)
\end{equation}
and
\begin{equation}\label{eq:det3}
\det\left(B'_{k}+C'_{k}-2A'_{k}\right)=
\det\left(B'_{k-1}+C'_{k-1}-2A'_{k-1}\right)
\det \left( 3\left(B'_{k-1}-C'_{k-1}\right)\right)
\end{equation}
for $k\ge 1$. Furthermore observe that
\begin{equation}\label{eq:det4}
\det\left( 3\left(B'_{k-1}-C'_{k-1}\right)\right)=3^{2^{k}}\det\left(C'_{k-1}-B'_{k-1}\right)
\end{equation}
as the involved matrices are of size $2^k\times 2^{k}$. Let 
$$z_k=\det A'_k, \ u_k=\det \left(C'_k-B'_k\right),
\ v_k=\det \left(B'_k+C'_k-2A'_k\right).
$$
By explicit calculations we get the initial values $z_0=1,u_0=-1, v_0=3.$
From (\ref{eq:det1})--(\ref{eq:det4}) we obtain the following recursions for
$k\ge 1$
\begin{equation}\label{def:zk}
z_{k}=z_{k-1}u_{k-1},\qquad u_{k}=u_{k-1}v_{k-1},\qquad v_{k}=3^{2^{k}}v_{k-1}u_{k-1}.
\end{equation}
By combining the second and the third equation in (\ref{def:zk}) we obtain
$v_k=3^{2^k}u_k$ which allows to eliminate $v_k$.
We obtain the new system of recursions 
\begin{equation}\label{def:zknew}
z_k=z_{k-1}u_{k-1},\qquad u_k=3^{2^{k-1}}u^2_{k-1},\qquad k\ge 1.
\end{equation}
 This already implies that all matrices
$A'_k$ are regular, but for the sake of completeness we provide the solution
for the recursion above. It is not hard to show that 
$$u_k=3^{k2^{k-1}}, \quad z_k=3^{(k-2)2^{k-1}+1}
$$
provides a solution to the system (\ref{def:zknew}) with the
initial conditions $z_0=1$ and $u_0=-1$.
As a consequence thereof we get that 
\[
\det A'_{d+1}=3^{d2^{d+1}+1}.
\]
Note that $M'_d$ differs from  $A'_{d+1}$ only in its additional row and
additional column. The additional column (the third column) of $M'_d$ 
corresponds to the third unit vector. By developing the determinant 
of $M'_d$ with respect to this column, we obtain
$$
\det M'_d=\det A'_{d+1}=3^{d2^{d+1}+1},
$$
which implies that $M'_d$ is regular and hence 
$\text{rank}(M_d)\geq 2^d+1$.
\end{proof}

Let us mention that one can show that $\text{rank}\left(M_d\right)=2^d+1$
by calculating the reduced row echelon form of matrix
$M_d$. For our purposes it sufficed to show a weaker 
result which could be obtained more elegantly.

\subsection{The COVP for the general case: ${(d,s)}$-AP}

Theorem~\ref{thm:axial} and Theorem~\ref{thm:planar} impose the following
question for the $(d,s)$-AP.

\begin{question}\label{conj1}
Is it true that a feasible instance of the $(d,s)$-AP with cost array $C$ 
has the COVP if and only if  $C$ is
a sum-decomposable array with parameters $d$ and $s$?
\end{question}

Theorem~\ref{thm:axial} and Theorem~\ref{thm:planar} imply that the answer to Question~\ref{conj1}
is affirmative in the following cases: $(2,1)$-AP, $(3,1)$-AP, $(3,2)$-AP, 
$(4,1)$-AP and
$(4,3)$-AP. This leaves us with the $(4,2)$-AP as the smallest unsettled
case. This is also the smallest case for which it is not guaranteed that a 
feasible solution exists for all $n\ge 2$. 

The following example shows that the answer to Question~\ref{conj1} is negative in general.

\begin{example}\label{count}
There are 72 Graeco-Latin squares of size 3, hence there are 72 feasible 
solutions for the $(4,2)$-AP with $n=3$, see Observation~\ref{obs3}. We consider the system of linear equations that is obtained by requiring
that all 72 feasible solutions have the same objective value.
The dimension of the solution space of this system of equations, and thus the
dimension of the space of cost arrays with the COVP, is 49, which can
easily be calculated by a computer algebra system. 
By  Proposition~\ref{prop:SAVS} one gets that the dimension 
of {\sc SAVS}(4,2,3) is 33. Hence, there exists a cost array with the COVP that is 
not sum-decomposable. Now we provide one such array.

Let $C$ be the $3\times 3\times 3 \times 3$ array where
$c(1,1,1,2)$, $c(1,1,2,1)$, $c(1,2,1,1)$, $c(1,2,2,2)$, $c(2,1,1,1)$,
$c(2,1,2,2)$, $c(2,2,1,2)$, $c(2,2,2,1)$ and $c(3,3,3,3)$ have value 1 and 
all other entries have value 0. All 72 feasible solutions of the (4,2)-AP 
with this cost array have the objective value 1, and it is  easy to check
that $C$ is not sum-decomposable.
\end{example}

We did not find counterexamples for the $(4,2)$-AP for $n\ge 4$. For $n=4$ and $n=5$ the computer calculations gave the affirmative answer to Question~\ref{conj1}. For $n=6$ there are no feasible
solutions and for $n=7$ the number of feasible solutions gets too large to handle.

\begin{conjecture}\label{conj2}
A feasible instance of the $(4,2)$-AP with cost array $C$ of size $n\neq 3$ 
has the COVP if and only if  $C$ is
a sum-decomposable array with parameters $4$ and $2$.
\end{conjecture}

We believe that Example~\ref{count} occurs since for $n=3$ the number of feasible solutions is relatively small, but then grows very fast. 
Note that for larger values of $n$ even the number of Graeco-Latin squares is unknown. This eliminates explicit proof approaches as the set of feasible solutions is not known. A proof would need to exploit the structure of the set of Graeco-Latin squares.

We checked that the answer to Question~\ref{conj1} for the
$(5,2)$-AP for $n=4$ is affirmative. For $n=2,3,6$ there are no feasible solutions.
For $n=5$ the set of feasible solutions became too large for our straightforward
computational approach. The same happened for the $(5,3)$-AP for $n=4$, and for $n=2,3$ the problem is again infeasible. 
The motivation behind our experiments was our wish to obtain a feeling whether the answer to Question~\ref{conj1} is affirmative for sufficiently large $n$. We believe so, but we could handle only very small cases and do not have enough empirical results to propose a conjecture.


\section{The COVP for $d$-dimensional transportation\\
problems}\label{section:transportation}

In this section we deal with the COVP for $d$-dimensional transportation 
problems. Specifically, we show that our COVP characterization for the
axial $d$-dimensional assignment problem carries over to the
axial $d$-dimensional transportation problem while this approach fails for the
more involved planar case.

Multi-dimensional transportation problems are known in the literature 
under diverse names. Alternative names are for example multi-index or $d$-index transportation problems, $d$-fold transportation problems and 
multi-way or $d$-way transportation problems, see e.g.~\cite{DLO06,QS09}.

The $d$-dimensional transportation problem can be defined along the lines of
the definition of the $d$-dimensional assignment problem $(d,s)$-AP.
We are given a $d$-dimensional  $n_1\times n_2 \times \cdots \times n_d$ cost 
array $C$. While in the assignment case the right hand side of all equality 
constraints is equal to one, in the transportation case we are additionally 
given an $s$-dimensional array $B^Q$ for each set $Q \in {\mathcal Q}_s$ 
of fixed indices which provides the right hand side values for this group of
constraints induced by the set $Q$. 
The arrays $B^Q$ can be viewed as marginals for the transportation
array $X=(x(i_1,i_2,\ldots,i_d))$. We refer to the resulting transportation
problem as $(d,s)$-TP.

Like for the assignment case, we obtain the {\em axial
 $d$-dimensional transportation problem\/} when $s=1$ and the 
{\em planar $d$-dimensional transportation problem\/} when $s=d-1$.
As we will deal with the axial $d$-dimensional transportation problem below, we
provide its explicit formulation. 

We are given an
$n_1\times n_2\times \cdots\times n_d$ cost array $C=(c(i_1,i_2,\ldots,i_d))$
and $d$ supply-demand vectors $B_1,\ldots,B_d$, where the $k$-th vector 
$B_k=(b_k(i))$ is an $n_k$-dimensional vector over the nonnegative integers.
Furthermore we assume $\sum_{i=1}^{n_1} b_1(i)=
\sum_{i=1}^{n_2}b_2(i)=\cdots =\sum_{i=1}^{n_d}b_d(i)$. 
Let $I_r=\{1,\ldots,n_r\}$ be the index set for $i_r$, $r=1,\ldots,d$.
We obtain the following formulation for the $(d,1)$-TP:
\begin{align*}
\min&\sum_{i_1\in I_1}\sum_{i_2\in I_2}\ldots\sum_{i_d\in I_d} 
c(i_1,i_2,\ldots,i_d)x(i_1,i_2,\ldots,i_d)\\
\text{s.\@t. }&\sum_{\substack{i_1\in I_1,\ldots,i_d\in I_d\\ \text{s.t. } i_k=j}
}x(i_1,i_2,\ldots,i_d)=b_k(j) \text{ for all }
k\in\{1,\ldots,d\}, j\in \{1,\ldots, n_k\}\\
&x(i_1,i_2,\ldots,i_d)\ge 0 \hspace{75pt} \text{ for all }
i_r=1,\ldots,n_r,\ r=1,\ldots,d.
\end{align*}

If $X=(x(i_1,\ldots,i_n))$ has to be integral, the problem above becomes 
NP-hard for $d\ge 3$. For $d=2$ the well-known classical Hitchcock
transportation problem arises.

\begin{theorem}\label{thm:transp}
An instance of the axial $d$-dimensional transportation problem with cost array
$C$ has the COVP if and only if $C$ is sum-decomposable array with parameters $d$ and 1.
\end{theorem}

\begin{proof}
Any instance of the integral axial $d$-dimensional transportation problem can 
be transformed into an equivalent instance of the 
axial $d$-dimensional assignment problem. 
To that end, we replace every supply/demand facility that has a supply/demand 
value $t>1$ by $t$ facilities with identical transportation costs that have 
supply/demand value $1$. In this manner we get an equivalent problem with a 
blown up $n\times n\times\cdots\times n$ cost array where 
$n=\sum_{i=1}^{n_1}b_1(i)$ and all supplies/demands are $1$.
Thus the newly obtained problem is the $(d,1)$-AP.

For the integral version of the $(d,1)$-TP we can apply the COVP
characterization from Theorem~\ref{thm:axial} directly.
For the non-integral version observe that the
transformed problem with unit supplies and demands is a relaxation of the
$(d,1)$-AP which results if the integrality constraints on $X$ are dropped.
In this case it follows from Theorem~\ref{thm:axial} that the set of 
instances with the COVP is a subspace of {\sc SAVS}$(d,1,n)$, and is hence equal to {\sc SAVS}$(d,1,n)$. 
Note that the transformation that blows up the cost array and the inverse
transformation preserve the sum-decomposability property of the cost
array. 
\end{proof}

Note that setting $d=2$ in Theorem~\ref{thm:transp} implies that an 
instance of the classical transportation problem with cost matrix $C$ 
has the COVP if and only if $C$ is a sum matrix. The proof of
Theorem~\ref{thm:transp} provides the connection to assignment problems and
further to Berenguer's COVP characterization for the TSP, cf.\
Theorem~\ref{thm:tsp}. As a by-product this reveals the nature of the 
connection between results of Klinz and Woeginger~\cite{KW11} on the
optimality of the North-West corner rule and
Theorem~\ref{thm:transp}, and thus answers an open problem mentioned 
in the concluding section of \cite{KW11}.

At first sight one might expect that Theorem~\ref{thm:planar} for the 
planar $d$-dimensional assignment problem $(d,d-1)$-AP carries over to the
planar $d$-dimensional transportation problem $(d,d-1)$-TP. However several
difficulties arise in this case. First, note that the blow-up technique to transform the
transportation problem to a (continuous) assignment problem does not 
work in general in the planar setting. The second and probably bigger obstacle 
to a COVP characterization for the
planar case comes from the fact that for $d\ge 3$ the $d$-dimensional 
planar transportation problem does not necessarily have feasible solutions 
(not even in the non-integral case). Due to the universality result of 
de Loera and Onn~\cite{DLO06} checking feasibility for the 3-dimensional planar
(integer) transportation problem is as hard as deciding whether a general 
linear (integer) program has a feasible solution (the result already holds for
a fixed third dimension, i.e., for $n_3=3$). As the number of feasible solutions of a feasible instance of the
$d$-dimensional transportation problem can be as small 
as one, even in the
non-integral case, it is not any longer necessary for the COVP that all
dual constraints have to be fulfilled with equality. Hence 
the approach based on the complementarity slackness 
condition that works for the linear assignment problem and the classical 
transportation problem, that was explained in the introduction, fails for $d\ge 3$.

Concluding, there  does not seem to be much hope to be able to provide a nice 
sufficient and necessary condition for the set of instances with the 
COVP for the 3-dimensional
planar transportation problem and even less hope for cases with $d>3$.


\section{The COVP for spanning tree, shortest path and 
matching problems}  \label{section:other}

In this section we provide COVP characterization 
for  the minimum spanning tree problem, the
shortest path problem and the minimum weight maximum cardinality 
matching problem.

\subsection{ The COVP for the minimum spanning tree problem} 

In the minimum spanning tree problem (MST)  we are given a connected,
undirected graph $G=(V,E)$ and edge weights $w_e$ for each edge $e\in E$. The task is to find a spanning tree for which the sum of edge
weights is minimal.

\begin{lemma}\label{lem:mst}
Let $I$ be an instance of the MST with graph $G$ and weights $w=(w_e)$. If $I$ has the COVP,
 then every edge in any \textup{(}simple\textup{)} cycle in $G$ has the same weight.
\end{lemma}

\begin{proof}
Let $C$ be a (simple) cycle in $G$ and $e$ be an arbitrary edge from $C$. There exists a spanning tree $T$ which
contains all edges of $C$ except $e$. By adding edge $e$ to $T$ and removing 
from $T$ in turn an arbitrary edge $f\ne e$ from $C$,
we obtain another spanning tree
$T'$. As the weights of $T$ and $T'$ are identical, it follows that
$w_e=w_f$. Hence all edges in $C$ have the same weight.
\end{proof}

To formulate the COVP characterization for the MST we need the following
definition.

\begin{definition}
Let $G=(V,E)$ be an undirected graph. The undirected graph $H=(V_H,E_H)$
which has a vertex $v_e$ for each edge $e\in E$ and an edge 
$\{v_e,v_f\}\in E_H$ if and only if $e$ and $f$ lie on a common simple 
cycle $C$ is called {\em cycle graph\/} of $G$.
\end{definition}

\begin{theorem}\label{theo:mst}
Let $I$ be an instance of the MST problem with graph $G$ and weights $w=(w_e)$. Let $H$ be the cycle graph of $G$ and let $V_1,\ldots,V_{\ell}$ be the vertex sets of connected components of $H$ and $E_1,\ldots,E_\ell$ be the corresponding sets of edges in $G$.
Then $I$ has the COVP if and only if there exist constants $\alpha_i$, $i=1,\ldots,\ell$ such that for all $e\in E_i$ $w_e=\alpha_i$, for all $i$.
\end{theorem}

\begin{proof}
To prove that the stated condition is sufficient, let $T$ and $T'$ 
be two spanning trees. It is easy to see that one can
move from $T$ to $T'$ by a sequence of moves which add an edge $e\in
T'\setminus T$ and
delete an edge $f\in T\setminus T'$ where $f$ lies on the unique cycle in 
$T\cup \{e\}$. As all edges on the cycle have the same weight, it follows 
from iterative application that $w(T)=w(T')$, where $w(T)$ denotes the sum of all $w_e$, $e\in T$.

To prove necessity of the stated condition, first observe that 
every bridge of $G$ (corresponds to an isolated 
vertex in $H$) is part of every
spanning tree and hence can have arbitrary weight. The claim now follows by
applying Lemma~\ref{lem:mst} for each cycle $C$ in G. 
\end{proof}

Note that if $H$ is connected (which is the case for example if $G$ is
$2$-connected), then the COVP holds if all edges have the same weight.

It is easy to see that the COVP characterization for the MST problem can be
carried over to the setting of matroids (circuits play the role of cycles and
bases play the role of spanning trees).

\subsection{The COVP for the shortest path problem} 

Given a weighted graph (undirected or directed) with the vertex set $V=\{1, 2, ..., n\}$ the shortest 
path problem is the problem of finding a path from vertex 1 to vertex $n$ 
such that the sum of edge weights along the path is minimized. In what follows we consider both the undirected and the directed version of 
the shortest path problem in a complete graph 
and provide COVP characterizations.

\begin{theorem}\label{theo:undirsp}
Let $G=(V, E)$ be the complete undirected graph with the vertex set $V=\{1,2,\ldots,n\}$, $n\ge 3$, and let 
$w(i,j)$ denote the nonnegative weight of the edge $(i,j)$. 
This instance of the undirected
shortest path problem has the COVP if and only if the weights are 
of the following form
\begin{equation}\label{shortest-undirected}
w(i,j)=w(j,i)=\begin{cases}
a&\text{if }i=1, j\neq n,\\
b&\text{if }i\neq 1, j=n,\\
a+b&\text{if }i=1,j=n,\\
0&\text{otherwise}
\end{cases}
\end{equation}
for some non-negative reals $a$ and $b$. 
\end{theorem}

\begin{proof}
Assume that every path from 1 to $n$ has the same weight. For $n=3$ the result
is straightforward. Assume $n\ge 4$ and take two distinct vertices 
$i$ and $j$ such that $1<i,j<n$. Consider the five paths from vertex 1 to 
vertex $n$ that only go through a subset of the vertices  $\{1,i,j,n\}$. By
assumption we get the following relations
\begin{align*}
	w(1,n)&=w(1,i)+w(i,j)+w(j,n)\\
	&=w(1,j)+w(j,i)+w(i,n)\\
	&=w(1,i)+w(i,n)\\
	&=w(1,j)+w(j,n).
\end{align*}
By adding and subtracting appropriate equations we get that $w(i,j)=0$,
$w(1,i)=w(1,j)$, $w(i,n)=w(j,n)$, so \eqref{shortest-undirected}
follows. 

Note that the converse trivially holds,  which concludes the proof.
\end{proof}

\begin{theorem}\label{theo:dirsp}	
Let $G=(V,E)$ be the complete directed acyclic graph with the vertex set
$V=\{1,2,\ldots,n\}$ and edge set $E=\{(i,j)\in V\times V\colon i<j\}$, and let $w(i,j)$ denote the weight of edge $(i,j)$. This instance of the directed
shortest path problem has the COVP 
if and only if there exists a real vector $A=(a_i)$ such that 
\begin{equation}\label{shortest-directed}
	w(i,j)=a_j-a_i \qquad\mbox{for all~~} i,j\in\{1,\ldots,n\},\  i<j.
\end{equation}
\end{theorem}

\begin{proof}
Assume that every path from vertex 1 to vertex $n$ has the same
weight. Consider the path composed of
edges $(1,i)$, $(i,j)$ and $(j,n)$. It has the same weight as the path 
composed  of edges $(1,j)$ and $(j,n)$. It follows that
$w(i,j)=w(1,j)-w(1,i)$. Set $a_i:=w(1,i)$ 
for $i=1,\ldots,n$, so that $w(i,j)=a_j-a_i$.

Now assume that for all $i<j$ the weight of $(i,j)$ can be represented as in
\eqref{shortest-directed} for some vector $A=(a_i)$. Consider an arbitrary path from
vertex 1 to $n$, and let $1=v_1<v_2<\cdots<v_k=n$ be all vertices on that
path. Then the weight of the path is 
\[
\sum_{i=1}^{k-1}w(v_i,v_{i+1})=\sum_{i=1}^{k-1}a_{v_{i+1}}-a_{v_i}=a_{v_k}-a_{v_1}=a_n-a_1.
\]
Since this number is independent of the choice of path,  \eqref{shortest-directed} is also sufficient and hence the statement holds.
\end{proof}

\subsection{The COVP for the minimum weight maximum cardinality 
matching problem} 

In the minimum weight maximum cardinality matching problem 
we are given an undirected graph $G=(V,E)$ and edge weights $w(i,j)$ for each edge $(i,j)\in E$. Our goal is to
find a matching for which the sum of edge weights is minimal among all matchings of maximal cardinality.

\begin{theorem}\label{theo:match}
Let $I$ be an instance of the minimum weight maximum cardinality matching on the
complete undirected graph $G$ with $n$ vertices and edge weights $w(i,j)$.
\begin{itemize}
\item[(i)] If $n$ is odd, $I$ has 
the COVP if and only if all edge weights are equal.
\item[(ii)] If $n$ is even, $I$ has the COVP if and only if 
there exists a real vector $A=(a_i)$ such that 
\begin{equation}\label{matchings}
w(i,j)=a_i+a_j \qquad\mbox{for all~} i\neq j. 
\end{equation}
\end{itemize}
\end{theorem}

\begin{proof}

Let $n$ be odd. Suppose that every maximum cardinality 
matching has the same weight.
 Let $i,j,k\in V$ be three distinct vertices. Let $M$ be a maximum cardinality 
matching on the vertex set $V\setminus\{i,j,k\}$. By adding an arbitrary 
edge from the triangle defined by $i,j$ and $k$ to $M$ we obtain a 
maximum cardinality matching on the initial instance. 
By assumption it follows that every edge in the triangle defined by 
$i,j$ and $k$ has the same weight. Hence the statement follows.

Let $n$ be even. Assume that every perfect matching has 
the same weight. 
Since each of the two pairs of edges $(i,j)$, $(k,l)$ and $(i,l)$, $(j,k)$
can be identically extended to a perfect matching 
it follows that
\[w(i,j)+w(k,l)=w(i,l)+w(j,k)\]
for all distinct $i,j,k,l$. Hence, there exist two real 
vectors $U=(u_i)$ and $V=(v_i)$ such that $w(i,j)=u_i+v_j$ for all $i\neq j$. 
Since the weight matrix has to be symmetric ($G$ is undirected), 
there exists a real 
vector $A=(a_i)$ such that 
\eqref{matchings} holds. 

Note that \eqref{matchings} is clearly a sufficient condition for the
COVP.
\end{proof}


\section{Conclusion}

Our goal was to characterize the set of instances with the constant objective value property (COVP), i.e.\@ to investigate the space of all instances for which every feasible solution has the same objective value. The COVP is closely connected to the notion of admissible transformations, a topic studied by various authors.

As our central result, we showed that the COVP instances of the planar and the axial $d$-dimensional assignment problem are characterized by sum-decomposable arrays with the corresponding parameters. We provided a counterexample which shows that these results do not carry over to general $d$-dimensional assignment problem. The following remains a challenging open question: Does sum decomposability characterize the COVP instances in all cases of the multidimensional assignment problems for which feasible solutions exist and size of the instance is sufficiently large?

We used the results for the axial $d$-dimensional assignment problem to characterize the COVP instances for the axial $d$-dimensional transportation problem. 

Furthermore, as simpler side results, we characterized the COVP instances for the following classical combinatorial optimization problems: the minimum spanning tree, the shortest path problem in undirected and directed graphs and the minimum weight cardinality matching problem in complete graphs.

\section*{Acknowledgement}
This research was supported by the Austrian Science Fund (FWF): 
W1230, Doctoral Program ``Discrete Mathematics''.

The authors are grateful to Gerhard J. Woeginger for fruitful discussions on
the COVP topic and for suggesting a simplification of our original COVP
characterization for the 3-dimensional planar assignment problem.


\bibliographystyle{plain}

\end{document}